\documentclass[12pt]{amsart}
\usepackage[T1]{fontenc}
\usepackage{latexsym} 
\usepackage{amsfonts, amsthm, amsmath,amssymb}
\usepackage{times}
\usepackage{color}
\usepackage[all]{xy}
\xyoption{2cell}

\def\<{{\langle}} 
\def\>{{\rangle}}

\def\note#1{{}}

\def\note#1{} 

\def\cO{{\mathcal O}}

\def\lend#1#2{{{\rm End}\sb{#1}(#2)}}

\def\beq{\begin{equation}} 
\def\eeq{\end{equation}}

\def\id{\mathrm{id}} 
\def\im{{\rm Im}}

\newcounter{zlist} 

\newcounter{blist} 
 
\newcounter{rlist} 
 
 \def\stac#1{\raise-.2cm\hbox{$\stackrel{\displaystyle\otimes}{\scriptscriptstyle{#1}}$}}

\def\cten#1{\raise-.2cm\hbox{$\stackrel{\displaystyle\widehat{\otimes}}{\scriptscriptstyle{#1}}$}}

\def\Label#1{\label{#1}\ifmmode\llap{[#1] }\else 
\marginpar{\smash{\hbox{\tiny [#1]}}}\fi} 
\def\Label{\label} 
\newtheorem{proposition}{Proposition}[section]

\theoremstyle{definition} 
\newtheorem{definition}[proposition]{Definition}

\theoremstyle{remark} 
\newtheorem{remark}[proposition]{Remark} 
 
\newcounter{c} 
 
\newcommand{\etyk}[1]{\vspace{-7.4mm}$$\begin{equation}\Label{#1} 
\addtocounter{c}{1}} 
\renewcommand{\]}{\ifnum \value{c}=1 $$\else \end{equation}\fi} 
\def\ZZ{{\mathbb Z}}

\newcommand{\Cc}{\mathcal{C}}

\def\*C{{}^*\hspace*{-1pt}{\Cc}}
\def\text#1{{\rm {\rm #1}}}

\def\1{\mathbf{1}}

\newcounter{mnotecount}[section]
\renewcommand{\themnotecount}{\thesection.\arabic{mnotecount}}
\newcommand{\mnote}[1]
{\protect{\stepcounter{mnotecount}}$^{\mbox{\footnotesize
$
\bullet$\themnotecount}}$ \marginpar{
\raggedright \tiny\em
$\bullet$\themnotecount: #1} }

\begin{document} 

\title{Twisted reality condition for Dirac operators} 
\author[T.\  Brzezi\'nski]{Tomasz Brzezi\'nski}
\address{Department of Mathematics, Swansea University, Swansea SA2 8PP, U.K.\ \newline 
\indent Department of Mathematics, University of Bia{\l}ystok, K.\ Cio{\l}kowskiego  1M,
15-245 Bia\-{\l}ys\-tok, Poland}
\email{T.Brzezinski@swansea.ac.uk}   

\author[N.\ Ciccoli]{Nicola Ciccoli}
\address{Dipartimento di Matematica e Informatica, Universit\'a di Perugia, via Vanvitelli 1,  I-06123 Perugia, Italy}
\email{ciccoli@dipmat.unipg.it}
 
\author[L.\ D\k{a}browski]{Ludwik D\k{a}browski}
\address{SISSA (Scuola Internazionale Superiore di Studi Avanzati), Via Bonomea 265, 34136 Trieste, Italy} 
\email{dabrow@sissa.it} 

\author[A.\ Sitarz]{Andrzej Sitarz} 
 \address{Institute of Physics, Jagiellonian University,
prof.\ Stanis\l awa \L ojasiewicza 11, 30-348 Krak\'ow, Poland.\newline\indent
 Institute of Mathematics of the Polish Academy of Sciences,
\'Sniadeckich 8, 00-950 Warszawa, Poland.}
\email{andrzej.sitarz@uj.edu.pl}   

\subjclass[2010]{58B34, 58B32, 46L87} 
\begin{abstract} 
Motivated by examples obtained from conformal deformations of spectral triples and a spectral triple
construction on quantum cones, we propose a new twisted reality condition for the Dirac operator.
\end{abstract} 
\maketitle 
 
\section{Introduction}
\setcounter{equation}{0}  
In \cite{Connes:1995} Connes proposed that real spectral triples should be understood as a noncommutative counterpart of spin manifolds. This proposal was further supported by establishing in \cite{Connes:2013} 
 a one-to-one correspondence between classical spin 
geometries and real spectral triples for commutative algebras of functions over manifolds. 
Yet, the number of examples of genuine noncommutative real spectral triples is limited. So far, 
the list includes finite spectral triples \cite{Kr98, PaSi98}, noncommutative tori \cite{Connes:1995} and their 
subalgebras: noncommutative Bieberbach manifolds \cite{OlSi13} and the quantum pillow \cite{BrzSi14},
in some cases with a full classification of inequivalent, irreducible and equivariant real spectral 
triples \cite{PaSi06, Ve, OlSi13}. All further examples of so-called $\theta$-twisted or isospectral 
deformations \cite{CoLa, Si01, Va} are essentially obtained by using the construction of real spectral triples over 
the noncommutative torus or over the Moyal plane \cite{GGBISV}. The example of the Podle\'s standard 
quantum sphere \cite{DaSi03} was the 
only non-flat case not obtained as a $\theta$-twist;  however,  it failed to satisfy some other natural conditions 
set for spectral geometries. 
In fact,  with the help of Poisson geometry, it was proved in \cite{Ha04} that if a noncommutative spectral triple 
is a deformation of the real spectral triple of functions on a 2-dimensional smooth manifold, then the 
underlying Riemannian manifold can only be either the flat torus or the round sphere. On the other hand, some interesting examples, such as the quantum group $SU_q(2)$ \cite{DLSSV} allow for
an {\em almost real} spectral triple, with a slightly modified order-one condition, which 
fails to be exact. 

In this note, we propose a modified definition of a real spectral triple, in which only the reality structure is generalised.
Thus in contrast to an interesting recent paper \cite{LaMa} we remain in the framework of {\em bona fide} 
spectral triples here. We show that this definition allows for
fluctuations of Dirac operators, which do not change the bimodule of one forms. Finally, we 
illustrate the theory with two examples, which motivate the proposed definition.

\section{Twisted reality}
\setcounter{equation}{0}
\subsection{Reality twisted by a linear automorphism}\label{sec.twist}
 
Let $A$ be a complex $*$-algebra and let $(H,\pi)$ be a (left) representation of 
$A$ on a complex vector space $H$. A linear automorphism $\nu$ of $H$ defines an 
algebra automorphism
$$
\bar{\nu}: \lend {}H \to \lend{}H, \qquad \phi\mapsto \nu\circ\phi\circ\nu^{-1}.
$$
The inverse of $\bar{\nu}$ is $\phi\mapsto \nu^{-1}\circ\phi\circ \nu$.
Since $\bar{\nu}$ is an algebra map, the composite  
$$
\pi^\nu: \xymatrix{ A \ar[rr]^-{\pi} && \lend{} H \ar[rr]^{\bar{\nu}} && \lend{} H}
$$
is an algebra map too, and hence it defines a new representation $(H,\pi^\nu)$ of $A$.
The map $\nu$ is  an isomorphism that intertwines  $(H,\pi)$ with $(H,\pi^\nu)$.

\begin{definition}\label{def.reality}
Let $A$ be a complex $*$-algebra, $(H,\pi)$ a representation of $A$, and let $D$ be a linear operator 
on $H$.  
Let $\nu$ be a linear automorphism of $H$, such that there exists an algebra automorphism  $\hat\nu:A\to A$ satisfying 
\begin{equation}\label{proper}
\pi \circ \hat\nu = \bar\nu\circ \pi = \pi^\nu.
\end{equation}
We say that the triple $(A,H,D)$ admits 
a {\em $\nu$-twisted real structure} if there exists an anti-linear map $J: H\to H$ such 
that $J^2 = \epsilon\, \id$, and, for all $a,b\in A$,
\begin{equation}\label{o0c}
 [\pi(a), J\pi(b)J^{-1}] =0, 
\end{equation}
\begin{equation}\label{to1c}
 [D,\pi(a)] J \bar{\nu}^2(\pi(b))J^{-1} = J \pi(b)J^{-1}[D,\pi(a)], 
\end{equation}
\begin{equation}\label{tc}
DJ\nu = \epsilon' \nu JD, 
\end{equation}
\begin{equation}\label{reg}
\nu J\nu = J, 
\end{equation}
where $\epsilon,\epsilon' \in \{+,-\}$. 

If $(A,H,D)$ admits a grading operator $\gamma: H\to H$, $\gamma^2 = \id$, $[\gamma,\pi(a)] =0$, for all $a\in A$, $\gamma D = -D\gamma$, 
and $\nu^2 \gamma = \gamma\nu^2 $,
then the twisted real structure $J$ is also required to satisfy
\begin{equation}\label{gc}
\gamma J = \epsilon''J\gamma,
\end{equation}
where $\epsilon''$ is another sign.
\end{definition}

This purely algebraic definition of twisted reality is motivated by and aimed at being applicable to 
spectral triples. In this case $H$ is a Hilbert space, and $\pi$ is a star representation of $A$ 
by bounded operators. 
The operator $D$ is assumed to be a (essentially) selfadjoint operator on the dense domain $H_0$ with a compact resolvent, and such that for every $a \in A$ the commutators $[D, \pi(a)]$ are bounded. The operator $J$ is antiunitary. The grading $\gamma$ (in case of an even spectral triple) is self-adjoint.
Besides the usual analytic properties of the data of the spectral triple, the map $\nu$ is then a 
(essentially) selfadjoint invertible operator defined on the domain containing $H_0$ and with the requirement that $\bar{\nu}$ maps $\pi(A)$ into bounded operators. More precisely, 
$\forall a\!\in\! A$, $\nu\pi(a)\nu^{-1}$ preserves $H_0$ and is bounded on $H_0$, and so extends to an element of $B(H)$.
Furthermore \eqref{tc} and \eqref{reg} should be understood as operators $H_0\to H_0$.
In such a case we shall say that a spectral triple 
admits a {\em $\nu$-twisted real structure},
or simply that it is a {\em $\nu$-twisted real spectral triple}. 

The signs $\epsilon, \epsilon',\epsilon''$ determine the $KO$-dimension 
modulo 8 in the usual way 
\cite{Connes:1995}.
Often \eqref{o0c} is called the {\em order-zero condition}, \eqref{to1c} is called the {\em twisted order-one condition},  while  we shall refer to \eqref{tc} as to the {\em twisted
$\epsilon'$-condition}, and to \eqref{reg} as to the {\em twisted regularity} (with the adjective twisted omitted if $\nu = \hbox{id}$).  Note that 
the twisted order-one condition \eqref{to1c} can be equivalently stated as 
\begin{equation}\label{to1cr}
 [D,\pi(a)] J \pi(\hat{\nu}(b))J^{-1} = J \pi(\hat{\nu}^{-1}(b))J^{-1}[D,\pi(a)].
 \end{equation}
Furthemore, the twisted order-one condition can be understood as right $A$-linearity as follows. Given a representation $\pi$ of a $*$-algebra $A$ in $H$, a linear automorphism $\nu$ of $H$ and an anti-linear automorphism $J$ of $H$, one can define a right $A$-module structure on $H$ by
 \begin{equation}\label{right}
 h\cdot a := J \bar\nu(\pi(a^*)) J^{-1}(h) = \pi^{J\circ\nu}(a^*)(h), \qquad \mbox{for all $a\in A$, $h\in H$}.
 \end{equation}
 We denote $H$ with this $A$-module structure by $H^{J,\nu}$. Then the first-order condition \eqref{to1c} simply states that, for all $a\in A$, the commutators $[D,\pi(a)]$ are right $A$-linear maps from $H^{J,\nu^2}$ to $H^{J,\id}$.

 The relation \eqref{proper} between $\hat\nu$ and $\nu$ can be also termed as ``$\nu$ implements an automorphism $\hat\nu$ in the representation $\pi$".

\begin{remark}\label{rem.proper}
One can contemplate a more general notion of the $\nu$-twisted reality defined as in Definition~\ref{def.reality} except the existence of $\hat\nu$ satisfying \eqref{proper}. Note, however, that if $(H,\pi)$ is a faithful representation, then $\pi$ is a monomorphism and hence $A$ is isomorphic to $\pi(A)$. Therefore, if in addition, 
\begin{equation}\label{image}
\im (\pi^\nu) \subseteq \im(\pi),
\end{equation}
 then the restriction of $\bar{\nu}$ to $\pi(A)$ defines a unique automorphism $\hat\nu$ of the algebra $A$ that satisfies \eqref{proper}. All the examples discussed in Section~\ref{sec.examples} come with faithful representations $\pi$ and automorphisms $\nu$ satisfying condition \eqref{image}.
 Finally, we note in passing that one can always work with a faithful representation of the algebra $A/{\rm ker}\pi$.
\end{remark}

\subsection{Twisted fluctuations}

In the setup of Definition~\ref{def.reality}, let $\Omega^1_D$ be a bimodule of one-forms:
$$ \Omega^1_D :=\left\{ \sum_i \pi(a_i)[D,\pi(b_i)], \,\,a_i,b_i\in A \right\}, $$
where the indices $i$ run over a finite set. 

The standard fluctuation of a spectral triple $(A,H,D)$ consists of adding to the 
Dirac operator $D$ a selfadjoint one-form $\alpha \in \Omega^1_D$. 
In the case of a real spectral triple, the fluctuated Dirac operator, $D$, becomes 
$D+\alpha+\epsilon' J\alpha J^{-1}$, where $\alpha+\epsilon' J\alpha J^{-1}$ 
is selfadjoint. For a $\nu$-twisted real spectral triple we set the fluctuated 
Dirac operator $D_\alpha$ to be:
$$ 
D_\alpha:= D + \alpha + \epsilon' \nu J \alpha J^{-1} \nu,
$$
with the requirement that $\alpha + \epsilon' \nu J \alpha J^{-1} \nu $ is selfadjoint. 
We shall often use the shorthand notation $\alpha '=  \nu J \alpha J^{-1} \nu$.

\begin{proposition}
If $(A,H,D)$ with $J\!\in\!\lend{} H$ is a  
 $\nu$-twisted real spectral triple, then  
$(A,H,D_\alpha)$ with (the same) $J$ is also a 
 $\nu$-twisted real spectral triple. 
If $(A,H,D)$ is even with grading $\gamma$,
then $(A,H,D_\alpha)$ is even with (the same) 
grading $\gamma$.
The composition of twisted fluctuations is a twisted fluctuation.
\end{proposition}
\begin{proof}
As a perturbation of $D$ by a bounded selfadjoint operator, the fluctuated Dirac operator 
$D_\alpha$ is selfadjoint, has bounded commutators with $\pi(a)\in A$ and has compact resolvent.

First, we shall demonstrate that a fluctuation of the fluctuated Dirac operator is also a fluctuation. 
In other words, the bimodule of one forms is independent of the choice of $\alpha$. 
Let us take any $a \in A$ and $ \alpha \in \Omega^1_D$, and  
compute:
$$
\begin{aligned}
\, [ \alpha' , \pi(a)] & = \nu J \alpha J^{-1} \nu \pi(a)  - \pi(a) \nu J \alpha J^{-1} \nu \\
& = \nu J \alpha J^{-1} \nu \pi(a)  - \nu(\nu^{-1}\pi(a) \nu) J \alpha J^{-1} \nu \\
& = \nu J \alpha J^{-1} \nu \pi(a) - \nu \pi (\hat\nu^{-1}(a) ) J \alpha J^{-1} \nu \\
& =  \nu J \alpha J^{-1} \nu \pi(a) - \nu J \alpha J^{-1} \pi (\hat\nu(a)) \nu 
\\
& =  \nu J \alpha J^{-1} \nu \pi(a) - \nu J \alpha J^{-1} \left(\nu  \pi (a) \nu^{-1}  \right) \nu
 = 0,
\end{aligned}
$$
by \eqref{proper}, \eqref{o0c} and \eqref{to1cr}. Therefore, for any $ \alpha \in \Omega^1_D$ and $a \in A$:
$$ [D_\alpha, \pi(a)] = [D,\pi(a)] + [\alpha, \pi(a)], $$
but then the last commutator is again a one-form.  
For this reason, the fluctuated Dirac operator $D_\alpha$ does satisfy the $\nu$-twisted order-one
condition (\ref{to1c}).
The statement about the grading $\gamma$ is easily verified.

To finish the proof, it remains only to check that $D_\alpha$ satisfies the compatibility
relation with $J$ (\ref{tc}), that is:
$$ D_\alpha J\nu =\epsilon' \nu JD_\alpha. $$ 
Since $D$ itself satisfies \eqref{tc}, suffices it to check it for $\alpha + \epsilon' \alpha'$:
$$
\begin{aligned}
 (\alpha + \epsilon' \alpha') J \nu &= \alpha J \nu + \epsilon' \nu J \alpha J^{-1} \nu J \nu \\
 &=  \alpha J \nu  + \epsilon' \nu J \alpha \\
 &=  \epsilon' \nu J \left( \alpha + \epsilon' J^{-1} \nu^{-1} \alpha J \nu \right) \\
 &= \epsilon' \nu J \left( \alpha + \epsilon' \alpha' \right).
 \end{aligned}
$$
where we have used (\ref{reg}).
\end{proof}

\section{Examples}\label{sec.examples}

The definitions in Section~\ref{sec.twist}  are motivated by two classes of examples. 
The first class arises from conformally rescaled spectral geometries, which (albeit in a different 
setting) were first proposed by Connes and Tretkoff \cite{CoTr} and led to twisted spectral triples. 
The setting using spectral triples instead was proposed by D\k{a}browski and Sitarz 
\cite{DaSi15} for more general modifications. It appears, however, that for conformal rescaling 
the setup is quite robust, not only keeping most of the properties of spectral geometries (like 
the Hochschild cycle condition \cite{Si15}), but also making the conformally rescaled 
noncommutative tori into quantum metric spaces in the sense of Rieffel \cite{La15}.

The second class of examples comes from $q$-deformed geometries, where modular twisting 
automorphisms naturally appear and are significant for cyclic cohomology \cite{HaKr} and modular
Fredholm module constructions, as explained for example in \cite{ReSiYa13}. It is worth pointing out that the 
operator implementing the reality is intrinsically related to the Tomita-Takesaki operator, which is usually 
nontrivial in  $q$-deformed geometry.

\subsection{Conformally transformed Dirac operators}
\setcounter{equation}{0}

Let us assume that we have a real spectral triple $(A,H,D, J)$ with reality operator
$J$ and fixed signs $\epsilon, \epsilon'$. Let $k \in \pi(A)$ be positive and 
invertible such that $k^{-1}$ is also bounded, and let $k' := J k J^{-1}$.

\begin{proposition}\label{propconf}
If $(A,H,D,J)$ is a real spectral triple, which satisfies order-one condition, then for:
$$D_k =  k' D k', \qquad \nu(h) = (k^{-1} k') \, (h),$$
the triple $(A, H, J, D_k, \nu)$ is a $\nu$-twisted real spectral triple.
If furthermore $(A,H,D,J)$ is even with grading $\gamma$,
then $(A,H,D_k, J, \nu)$ is even with (the same) grading $\gamma$.
\end{proposition}
\begin{proof}
Since $k$ and $k'$ are bounded operators, it is clear that $\bar{\nu}$ maps bounded 
operators to bounded operators, and due to \eqref{o0c}, for all $a \in A$:
$$ \bar{\nu}(\pi(a)) = k^{-1} \pi(a) k. $$
Clearly $ \hat{\nu}(a) = (\tilde k)^{-1} a \tilde k $,
where $\tilde k\in A$ is such that $\pi(\tilde k)=k$, satisfies \eqref{proper}.

By the properties of the real spectral triple $(A,H,D,J)$, $k'(\mathrm{Dom}\,D)\subset \mathrm{Dom}\,D$ 
and the operator $D_k $ is obviously selfadjoint  on the domain $\mathrm{Dom}\,D$,
and due to \eqref{o0c} has bounded commutators with arbitrary $a \in A$. Furthermore, $D_k$ 
has compact resolvent: this is obvious if $D$ is invertible since then 
 $D_k^{-1}= k^{-1} D^{-1} k^{-1}$ and $D^{-1}$ is compact and $k^{-1}$ bounded by assumption.
Otherwise, we can employ some invertible bounded perturbation of $D$ on the kernel of $D$.

We show now that $D_k$ satisfies the twisted order-one condition \eqref{to1c}:
$$
\begin{aligned}
J \pi(b)J^{-1} & [D_k, \pi(a)]   = 
J \pi(b)J^{-1} JkJ^{-1} [D,\pi(a)] JkJ^{-1}
 \\
&=   k' [D,\pi(a)] k' J (k^{-2} \pi(b) k^{2})  J^{-1} 
  =  [D_k, \pi(a)] \,J \bar{\nu}^2(\pi(b)) J^{-1}.
\end{aligned}
$$
Next we check \eqref{reg}:
$$ \nu J \nu = k^{-1} J k J^{-1} J  k^{-1} J k J^{-1} = J. $$
Finally, if $JD = \epsilon' DJ$ then for $D_k$ we have:
$$ J D_k = J k' J^{-1} J D k' = \epsilon' k D J k' =\epsilon' k (k')^{-1} D_k (k')^{-1} k J,$$
so that \eqref{tc} is satisfied
$$ \nu J D_k =\epsilon' D_k J \nu. $$
The statement about the grading $\gamma$ is easily verified.
\end{proof}

\begin{remark} 
Clearly the KO-dimension of $(A, H, D_k, J, \nu)$ is the same as of $(A, H, D, J)$. In the `classical' 
case of a manifold $M$ and (commutative) $A=C^\infty(M)$ represented in the usual way on $L^2$-sections of some vector bundle on $M$ with $\mathrm{Ad}_J$ being the complex conjugation, the conformal twists $\nu$ are always trivial as $ J k J^{-1} = k$ for a positive $k$ and hence $\nu = \hbox{id}$.
\end{remark}

\begin{remark}\hspace{-5pt}\footnote{We thank Tomasz Maszczyk for raising this point.}\label{propconf2}
Using arguments similar to those in the proof of Proposition~\ref{propconf}, one can show that if  $(A,H,D, J, \nu)$ is a $\nu$-twisted real spectral triple which satisfies the twisted order-one condition, then, for all $k$ as in Proposition~\ref{propconf} such that $\bar{\nu}(k k') = k k'$,  $(A, H, D_k, J, \mu)$ is a $\mu$-twisted real spectral triple, where 
$$
D_k =  k' D k', \qquad \mu(h)  =  k' \nu k^{-1}\, (h).
$$
 The grading $\gamma$, if it exists, is again unchanged.
\end{remark}

\subsection{Twisted reality of Dirac operators on quantum cones}
\setcounter{equation}{0}
The coordinate algebra of the {\em quantum disc} $\cO(D_{q})$ is a complex $*$-algebra generated by $z$, subject to the relation
\begin{equation}\label{disc}
z^*z-q^2zz^* = 1-q^2,
\end{equation}
 where $q\in (0,1)$; see \cite{KliLes:two}. $\cO(D_q)$ can be understood as a $\ZZ$-graded algebra
 $$
 \cO(D_q)  = \bigoplus_{n\in \ZZ}\cO(D_q)_n,
 $$
 with the degrees given on the generators by $|z| = - |z^*| =1$. For any $N>1$, the {\em quantum cone} $\cO(C^N_q)$ is defined as the $*$-algebra generated by $y$ and self-adjoint $x$ subject to relations:
\begin{equation}\label{cone}
xy = q^{2N}yx, \qquad yy^* = \prod_{l=0}^{N-1}\left(1-q^{-2l}x\right), \qquad y^*y = \prod_{l=1}^{N}\left(1-q^{2l}x\right);
\end{equation}
see \cite{Brz:com}. The map  $\cO(C^N_q) \to \cO(D_q)$ defined on the generators as  $x\mapsto 1-zz^*$, $y\mapsto z^N$ provides an identification of the cone with the subalgebra of $\cO(D_q)$ consisting of all elements of degree a multiple of $N$, i.e.,
$$
\cO(C^N_q) \cong \bigoplus_{n\in \ZZ}\cO(D_q)_{nN}.
$$
Throughout we will view $\cO(C^N_q)$ as a subalgebra of $\cO(D_q)$ in this way. 

Using the $\ZZ$-grading of $\cO(D_q)$ we define a degree counting algebra automorphism, 
$$
\nu : \cO(D_q) \to \cO(D_q), \qquad a\mapsto q^{|a|} a,
$$
for all homogeneous elements of $\cO(D_q)$. Obviously, $\nu$ preserves the degrees of homogeneous elements, and hence restricts to an automorphism of $\cO(C^N_q)$. It is also compatible with the $*$-structure in the sense that, 
\begin{equation}\label{regularity}
\nu \circ * \circ \nu = * .
\end{equation}

The maps $\partial_-, \partial_+: \cO(D_q)\to \cO(D_q)$, defined on generators of the disc algebra by
\begin{equation}\label{partial}
\partial_- (z) = z^*, \quad \partial_- (z^*) = 0, \qquad \partial_+ (z) = 0, \quad \partial_+ (z^*) = q^{2} z,
\end{equation}
extend to the whole of $\cO(C^N_q)$ as $\nu^2$-skew derivations, i.e., by the twisted Leibniz rule,
\begin{equation}\label{Leibniz}
\partial_\pm (ab) = \partial_\pm(a) \nu^{2}(b) + a\partial_\pm (b),
\end{equation}
for all $a,b \in \cO(D_q)$. As a consequence of \eqref{partial} and \eqref{Leibniz}, the maps $\partial_\pm$ have degrees $\pm 2$, respectively, and since $\nu$ is the degree-counting automorphism,
\begin{equation}\label{q-skew}
\nu \circ \partial_\pm \circ \nu^{-1} = q^{\pm 2}\, \partial_\pm.
\end{equation}
The combination of \eqref{regularity}, \eqref{partial} and \eqref{Leibniz} also yields
\begin{equation}\label{pm}
\nu \left(\partial_\pm (a)^*\right) = \nu^{-1}\left(\partial_\mp(a^*)\right).
\end{equation}

Set
$$
H_+ = \bigoplus_{n\in \ZZ}\cO(D_q)_{nN+1}, 
\quad H_-= \bigoplus_{n\in \ZZ}\cO(D_q)_{nN-1}, \qquad H = H_+\oplus H_-,
$$
i.e., $H$ is the external direct sum of two vector spaces $H_+$ and $H_-$. Since $\cO(C^N_q)$  is spanned by homogeneous elements of degrees that are multiples of $N$, and $\nu$ preserves the degrees, both $H_\pm$ and hence also $H$ are left $\cO(C^N_q)$-modules (representations) by the action
\begin{equation}\label{action}
\pi(a) (h_\pm) = \nu^2(a)h_\pm, \qquad \mbox{for all $a\in \cO(C^N_q)$, $h_\pm \in H_\pm$}.
\end{equation}
Furthermore, $\nu$ restricts to an automorphism of the vector space $H$. Note, however, that since $\nu$ is an automorphism of the graded algebra $\cO(D_q)$ and the action of $\cO(C^N_q)$ is defined by the multiplication in  $\cO(D_q)$, the induced automorphism $\bar\nu$ of the (linear) endomorphism ring takes a particularly simple form on $\pi(a)$,
\begin{equation}\label{barnu}
\bar\nu(\pi(a)) = \pi(\nu(a)), \qquad \mbox{for all $a\in \cO(C^N_q)$,}
\end{equation}
so that $\hat\nu$ satisfying \eqref{proper} is simply the restriction of $\nu$ to $\cO(C^N_q)$. Since $\partial_\pm$ have degree $\pm 2$, respectively, 
\begin{equation}\label{delpm}
\partial_\pm(H_\mp) \subseteq H_\pm,
\end{equation}
and the $*$-operation reverses degrees, we can define operators
\begin{equation}\label{Dirac}
D: H \to H, \qquad (h_+,h_-)\mapsto \left(-q^{-1}\partial_+ (h_-),\ q\partial_-(h_+)\right),
\end{equation}
\begin{equation}\label{reality}
J: H\to H, \qquad (h_+, h_-)\mapsto \left(-h_-^*, h_+^*\right),
\end{equation}
and also the grading operator,
\begin{equation}\label{gamma}
\gamma:  H\to H, \qquad (h_+, h_-)\mapsto \left(h_+, -h_-\right).
\end{equation}

\begin{proposition} \label{lem.cone}
With these definitions, $(\cO(C^N_q), H, D)$ is an (algebraic) even 
spectral triple of KO-dimension two with 
$\nu$-twisted real structure $J$ and grading $\gamma$.
\end{proposition}

\begin{proof} The order zero condition \eqref{o0c} follows immediately by the definition of $J$ and the fact that $*$ is an anti-algebra involution, the condition \eqref{reg} is a straighforward consequence of \eqref{regularity}. Checking the twisted order-one condition \eqref{to1c} and twisted 
$\epsilon'$-condition
 \eqref{tc} is a bit more involved and uses all the properties \eqref{regularity}--\eqref{barnu}. Specifically, for all $h_\pm \in H_\pm$,
$$
\nu J D(h_\pm) =  q^{\pm 1} \nu\left(\partial_\mp (h_\pm)^*\right)
= q^{\pm 1} \nu^{-1}\left(\partial_\pm (h_\pm^*)\right),
$$
by \eqref{pm}. On the other hand,
$$
DJ\nu(h_\pm) =\pm D\left(\nu^{-1}(h_\pm^*)\right) =  q^{\mp 1} \partial_\pm\left(\nu^{-1}(h_\pm^*\right))
= q^{\pm 1} \nu^{-1}\left(\partial_\pm (h_\pm^*)\right),
$$
where the first equality follows by \eqref{regularity} and the third one by \eqref{q-skew}. This proves the twisted $\epsilon'$-condition \eqref{tc}. To prove the twisted order-one condition \eqref{to1c} observe that, for all $a,b\in \cO(C^N_q)$ and $h_\pm \in H_\pm$,
\begin{eqnarray*}
[D,\pi(a)] (h_\pm) &=& 
\pm q^{\pm 1} \left( \partial_\mp\left(\nu^2(a)h_\pm\right)  - \nu^2(a) \partial_\mp\left(h_\pm\right)\right)\\
&=&\pm q^{\pm 1} \partial_\mp\left(\nu^2(a)\right)\nu^2\left(h_\pm\right) 
= \pm q^{\pm 5} \nu^2\left(\partial_\mp(a)h_\pm\right),
\end{eqnarray*}
by the $q$-skew-derivation properties \eqref{Leibniz} and \eqref{q-skew}, and
$$
J\pi(b)J^{-1} (h_\pm) = h_\pm \nu^2(b)^* = h_\pm \nu^{-2}(b^*),
$$
by \eqref{regularity}.  In view of the latter and \eqref{barnu},
$$
J\bar\nu^2(\pi(b))J^{-1} (h_\pm) = h_\pm \nu^{-4}(b^*).
$$
Putting all this together one easily checks that $J$ and $D$ satisfy the twisted order-one condition \eqref{to1c}.
The statement about the grading $\gamma$ is easy to see and we have $J^2=\epsilon = -1$,
$\epsilon' = 1$, $\epsilon" = -1$, which corresponds to the KO-dimension 2.
\end{proof}

The above construction of the twisted reality structure for a Dirac operator is not particularly specific to the case of quantum cones, but is a special case of a more general situation. The key role is played by the fact that the module structure is induced from the multiplication of a graded $*$-algebra which admits $q$-skew derivations compatible with the $*$-structure. 

Let $G$ be a group and $A = \oplus_{g\in G}A_g$ be a $G$-graded $*$-algebra, and set $B:= A_e$, where $e$ is the neutral element of $G$. The grading is assumed to be compatible with the $*$-structure in the sense that, for all $g\in G$, 
$$
A_g^* \subseteq A_{g^{-1}},
$$
so, in particular, $B$ is a $*$-subalgebra of $A$. Let $G_+ \subset G$ and set
$$
G_- := \{g^{-1} \; |\; g\in G_+\}, \qquad H_\pm = \bigoplus_{g\in G_\pm} A_g, \qquad H = H_+\oplus H_-,
$$
i.e., $H$ is the external direct sum of two vector spaces $H_+$ and $H_-$. The definition of $G_-$ ensures that $H^*_\pm \subseteq H_\mp$. 
Let $\nu$ be a graded (i.e., degree preserving) algebra automorphism of $A$ that satisfies \eqref{regularity}. Let $\partial_\pm: A\to A$ be linear maps that satisfy \eqref{Leibniz}, \eqref{q-skew}, \eqref{pm} and \eqref{delpm}. If we view $H$ as a left $B$-module by the multiplication in $A$ or  as in \eqref{action}, then the operator $D$ defined by \eqref{Dirac} has the twisted reality structure $J$ given by \eqref{reality}
and the grading $\gamma$ given by \eqref{gamma}.

 \section*{Acknowledgments}
T.\ Brzezi\'nski acknowledges support of INdAM--GNFM and would like to thank SISSA for hospitality. 
N. Ciccoli acknowledges support of INdAM--GNSAGA. A.\ Sitarz thanks Universit\`a degli Studi di Perugia for hospitality. N. Ciccoli, L. D\k{a}browski and A. Sitarz acknowledge support of the NCN grant 2012/06/M/ST1/00169.  The authors are most grateful to the referees for very helpful comments.

\end{document}